\documentclass{amsart}

\input epsf
\def\relabelbox{%
  \hbox\bgroup%
}%
\def\endrelabelbox{%
\egroup%
}%
\def\relabel #1#2 {%
  \special{ps:/a {} def}%
  \smash{\rlap{#2}}%
}%
\def\adjustrelabel <#1,#2> #3#4 {%
  \special{ps:/a {} def}%
  \smash{\rlap{\kern #1 \raise #2\hbox{#4}}}%
}%
\def\extralabel <#1,#2> #3 {\smash{\rlap{\kern #1 \raise #2\hbox{#3}}}}%

\usepackage{epsfig}

\def\co{\colon\thinspace}
\newcommand{\R}{\mathbb{R}}
\newcommand{\Z}{\mathbb{Z}}
\newcommand{\rot}{{\tt rot}}
\newcommand{\ogamma}{\overline{\gamma}}
\newcommand{\bfoo}{{\mathbf 0}}

\newtheorem{thm}{Theorem}

\newtheorem{prop}[thm]{Proposition}

\theoremstyle{definition}
\newtheorem*{rem}{Remark}

\begin{document}

\title[The Whitney--Graustein Theorem]{A contact geometric proof
of\\the Whitney--Graustein Theorem}

\author{Hansj\"org Geiges}

\address{Mathematisches Institut, Universit\"at zu K\"oln,
Weyertal 86--90, 50931 K\"oln, F.R.G.}

\email{geiges@math.uni-koeln.de}


\begin{abstract}
The Whitney--Graustein theorem states that regular closed curves in the
$2$-plane are classified, up to regular homotopy, by their rotation number.
Here we give a simple proof based on contact geometry.
\end{abstract}

\maketitle

\section{Introduction}
A {\bf regular closed curve} in the $2$-plane is a continuously 
differentiable map
$\ogamma\co [0,2\pi ]\rightarrow\R^2$ with the following properties:
\begin{itemize}
\item[(i)] $\ogamma (0)=\ogamma (2\pi),\;\; \ogamma'(0)=\ogamma'(2\pi )$,
\item[(ii)] $\ogamma'(s)\neq \bfoo$ for all $s\in [0,2\pi ]$.
\end{itemize}
If we identify the circle $S^1$ with $\R /2\pi \Z$, we may think
of $\ogamma$ as a continuously differentiable map $S^1\rightarrow\R^2$.

The {\bf rotation number} $\rot (\ogamma )$ of $\ogamma$ is the
degree of the map
\[ \begin{array}{ccc}
S^1 & \longrightarrow & \R^2\setminus\{\bfoo\}\\
s   & \longmapsto     & \ogamma'(s).
\end{array} \]
In other words, $\rot (\ogamma )$ is simply a signed count of the
number of complete turns of the velocity vector $\ogamma'$
as we once traverse the closed curve~$\ogamma$, see
Figure~\ref{figure:immersions}.

\begin{figure}[h]
\centerline{\relabelbox\small
\epsfxsize 8cm \epsfbox{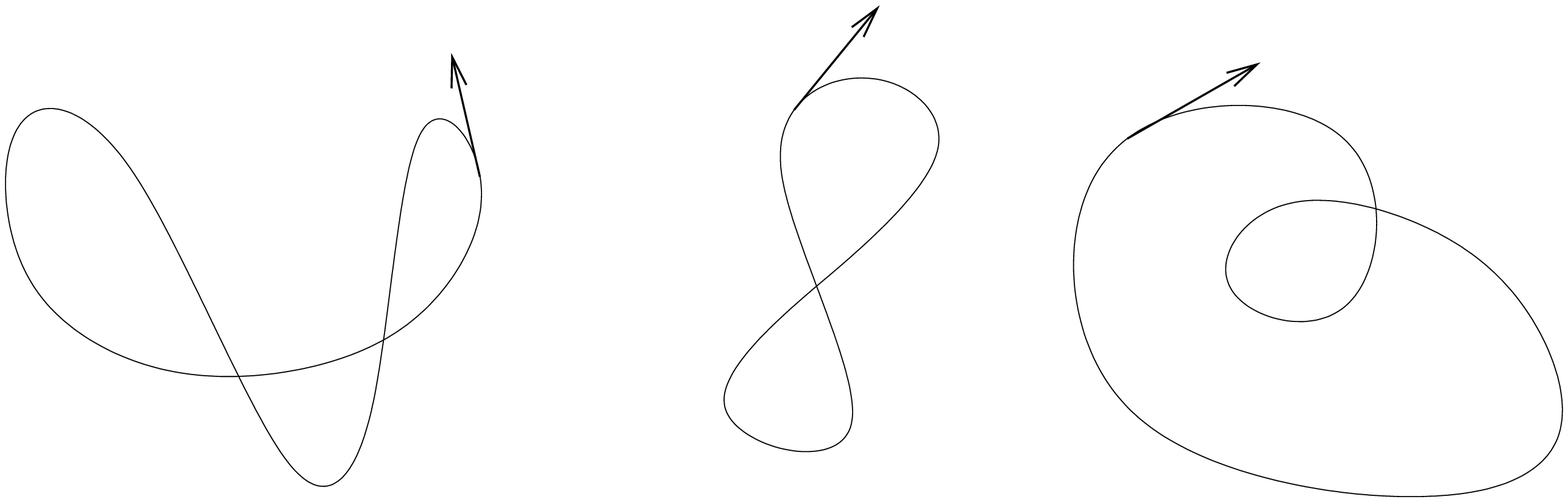}
\endrelabelbox}
\caption{Regular closed curves $\ogamma$ with $\rot (\ogamma )$ equal to
$1$, $0$, $-2$, respectively.}
\label{figure:immersions}
\end{figure}

A {\bf regular homotopy} between two such regular closed
curves $\ogamma_0,\ogamma_1$ is a continuously differentiable
homotopy via regular closed
curves $\ogamma_t\co S^1\rightarrow\R^2$, $t\in [0,1]$.
The rotation number clearly stays invariant under regular homotopies.
The following theorem is commonly known as the Whitney--Graustein theorem.
It was first proved in a paper by H.~Whitney~\cite{whit37},
who writes: `This theorem, together with its proof, was suggested to me by
W.~C.~Graustein.' For alternative presentations see
\cite[Chapter~0]{adac93} or \cite[p.~47 {\it et seq.}]{geig03}.

\begin{thm}
\label{thm:WG}
Regular homotopy classes of regular closed
curves $\ogamma\co S^1\rightarrow\R^2$
are in one-to-one correspondence with the integers, the correspondence
being given by $[\ogamma ]\mapsto\rot (\ogamma )$.
\end{thm}

Whitney's proof is elementary, but not without intricacies.
Here I want to
present a nonelementary proof --- based on contact geometry ---
where the geometric ideas are actually quite simple.

\begin{rem}
The modern terminology `regular homotopy' describes what
Whitney called a `deformation' of regular closed curves.
He seems to suggest, erroneously, that it is enough to require that
$\gamma_t(s)$ be continuous in $s$ and $t$ and a regular
closed curve for each fixed~$t$, but in the course of his
argument it becomes clear that he wants $\gamma'_t(s)$
to depend continuously on $t$ as well. Figure~\ref{figure:non-reghomotopy}
shows a homotopy of regular closed curves (first traverse the
big circle counter-clockwise,
then the small circle) with $\rot (\gamma_t)=2$ for
$t\in [0,1)$, but $\rot (\gamma_1)=1$.
\end{rem}

\begin{figure}[h]
\centerline{\relabelbox\small
\epsfxsize 8cm \epsfbox{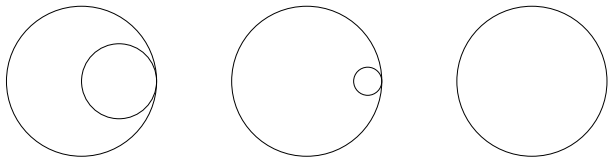}
\endrelabelbox}
\caption{A homotopy through regular closed curves with
$\rot$ not invariant.}
\label{figure:non-reghomotopy}
\end{figure}

\subsubsection*{Acknowledgement.}
The idea for the proof presented here
was inspired by a conversation with Yasha Eliashberg.
\section{Legendrian curves}
\label{section:Legendrian}
The {\bf standard contact structure} $\xi$ on $\R^3$, see
Figure~\ref{figure:ex1} (produced by Stephan Sch\"onenberger),
is the $2$-plane field
$\xi=\ker (dz+x\, dy)$. For a brief introduction to
contact geometry see~\cite{geig01}. No knowledge of contact geometry beyond
the concepts that I shall introduce explicitly will be required for
the argument that follows.

\begin{figure}[h]
\centerline{\relabelbox\small
\epsfysize 2truein \epsfbox{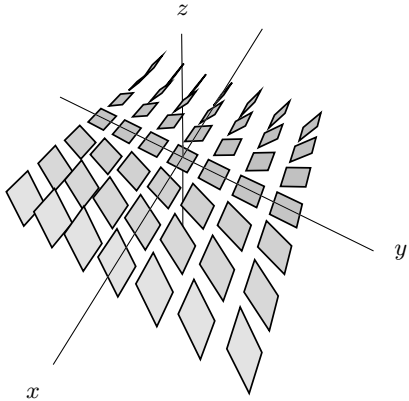}
\extralabel <-4.8cm,.1cm> {$x$}
\extralabel <.1cm,2.0cm> {$y$}
\extralabel <-2.8cm,5.2cm> {$z$}
\endrelabelbox}
\caption{The contact structure $\xi =\ker (dz+x\, dy)$.}
\label{figure:ex1}
\end{figure}

A regular closed, continuously differentiable
curve $\gamma\co S^1\rightarrow (\R^3 ,\xi )$ is
called {\bf Legendrian} if it is everywhere tangent to~$\xi$, that is,
$\gamma'(s)\in\xi_{\gamma (s)}$ for all $s\in S^1$.
When we write $\gamma$ in terms of coordinate functions
as $\gamma (s)=\bigl( x(s),y(s),z(s)\bigr)$, the condition for $\gamma$
to be Legendrian becomes $z'+xy'\equiv 0$.
The {\bf front projection} of $\gamma$ is the planar curve
\[ \gamma_{\mathrm F}(s)=\bigl( y(s),z(s)\bigr) ;\]
its {\bf Lagrangian projection}, the curve
\[ \gamma_{\mathrm L}(s)=\bigl( x(s),y(s)\bigr) .\]
Figure~\ref{figure:frontLag} shows the front and Lagrangian projection
of a Legendrian unknot in~$\R^3$.

\begin{figure}[h]
\centerline{\relabelbox\small
\epsfxsize 8cm \epsfbox{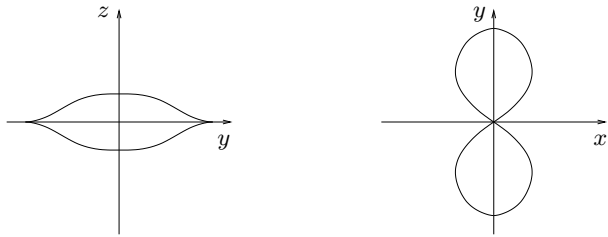}
\extralabel <-6.9cm,2.9cm> {$z$}
\extralabel <-1.9cm,2.9cm> {$y$}
\extralabel <-5.3cm,1.2cm> {$y$}
\extralabel <-0.3cm,1.2cm> {$x$}
\endrelabelbox}
\caption{A Legendrian unknot.}
\label{figure:frontLag}
\end{figure}

Notice that a Legendrian curve $\gamma$ can be recovered from
its front projection $\gamma_{\mathrm F}$, since
\[ x(s)=-\frac{z'(s)}{y'(s)}=-\frac{dz}{dy}\]
is simply the negative slope of the front projection. (Of course this
only makes sense for $y'(s)\neq 0$. Generically, the zeros of the
function $y'(s)$ are isolated, corresponding to isolated cusp points
where $\gamma_{\mathrm F}$ still has a well-defined slope.) Since $x(s)$ is
always finite, $\gamma_{\mathrm F}$ does not have any vertical
tangencies, and we can sensibly speak of left and right cusps.
These cusps are `semi-cubical'; a model is given by
$\bigl( x(s),y(s),z(s)\bigr) = (s,s^2/2,-s^3/3)$.

Likewise, $\gamma$ can be recovered from its Lagrangian projection
$\gamma_{\mathrm L}$
(unique up to translation in the $z$-direction),
for the missing coordinate $z$ is given by
\[ z(s_1)= z(s_0)-\int_{s_0}^{s_1}x(s)y'(s)\, ds.\]
Observe that the integral $\int xy'\, ds=\int x\, dy$, when
integrating over a closed curve, measures the oriented area enclosed by
that curve. Moreover, the Lagrangian projection $\gamma_{\mathrm L}$ of
a regular Legendrian curve $\gamma$ is always regular: if $y'(s)=0$,
the Legendrian condition forces $z'(s)=0$, and then the regularity of
$\gamma$ gives $x'(s)\neq 0$.

The idea for the proof of Theorem~\ref{thm:WG} is now the following.
Given a (regular closed) Legendrian curve $\gamma$ in $(\R^3,\xi )$,
one can assign to it an invariant (under Legendrian regular homotopies,
i.e.\ regular homotopies via Legendrian curves). This invariant
is likewise called `rotation number'. In fact, the
rotation number of $\gamma$  will be seen to
equal the rotation number of its Lagrangian projection~$\gamma_{\mathrm L}$.
Alternatively, the rotation number of $\gamma$ can be computed
from its front projection $\gamma_{\mathrm F}$, where it becomes
a simple combinatorial quantity (a count of cusps). Now, given
two regular closed curves $\ogamma_0,\ogamma_1$ in the plane with
equal rotation number, we can consider their lifts to Legendrian
curves $\gamma_0,\gamma_1$ (still with equal rotation number), and in the
front projection we can now `see', in a combinatorial way, a 
Legendrian regular homotopy between them. The Lagrangian projection of this 
Legendrian regular homotopy will give us the regular homotopy between
$\ogamma_0$ and $\ogamma_1$.

\section{The rotation number}

The plane field $\xi$ is spanned by the globally defined vector fields
$e_1=\partial_x$ and $e_2=\partial_y-x\,\partial_z$.
In terms of the trivialisation
of $\xi$ defined by these vector fields, we may regard the map $\gamma'$
(coming from a regular closed Legendrian curve $\gamma$) as a map
\[ \begin{array}{ccc}
S^1 & \longrightarrow & \R^2\setminus\{\bfoo\}\\
s   & \longmapsto     & \gamma'(s).
\end{array} \]
The {\bf rotation number} $\rot (\gamma )$ of a
Legendrian curve $\gamma$
is the degree of that map. This means that $\rot (\gamma )$ 
counts the number of rotations of the velocity vector $\gamma'$
relative to the oriented basis $e_1,e_2$
of $\xi$ as we go once around~$\gamma$. The rotation number is clearly
an invariant of Legendrian regular homotopies.

Under the projection $(x,y,z)\mapsto (x,y)$, each $2$-plane
$\xi_{\gamma (s)}$ maps isomorphically onto~$\R^2$, and the basis
$e_1,e_2$ for $\xi_{\gamma (s)}$ is mapped to the standard basis
$\partial_x,\partial_y$ for~$\R^2$. So
the following is immediate from the definitions.

\begin{prop}
\label{prop:rot-rot}
The rotation number of a (regular closed) Legendrian curve in
$(\R^3,\xi )$ equals the rotation number of its Lagrangian projection.
\qed
\end{prop}

A little more work is required to read off $\rot (\gamma )$ from the 
front projection~$\gamma_{\mathrm F}$. This, however, is well worth the effort, because
it turns the rotation number into a simple combinatorial quantity.

\begin{prop}
Let $\gamma$ be a (regular closed) Legendrian curve in $(\R^3,\xi)$.
Write $\lambda_+$ or $\lambda_-$, respectively, for the number of
left cusps of the front
projection $\gamma_{\mathrm F}$ oriented upwards or downwards; similarly
we write $\rho_{\pm}$ for the number of right cusps with
one or the other orientation. Finally, we write $c_{\pm}$
for the total number of cusps oriented upwards or downwards,
respectively. Then the
rotation number of $\gamma$ is given by
\[ \rot (\gamma )  =  \lambda_- -\rho_+ = \rho_- -\lambda_+ =
\frac{1}{2}(c_- - c_+).\]
\end{prop}

\begin{proof}
The rotation number $\rot (\gamma )$
can be computed by counting (with sign) how often the velocity
vector $\gamma'$ crosses $e_1=\partial_x$ as we travel once along~$\gamma$.

Since $x(s)$ equals the negative slope of the front projection, points
of $\gamma$ where the (positive) tangent vector equals $\partial_x$
are exactly the left cusps oriented downwards (see
Figure~\ref{figure:leftnegcusp}) and the right cusps oriented upwards.

\begin{figure}[h]
\centerline{\relabelbox\small
\epsfxsize 5cm \epsfbox{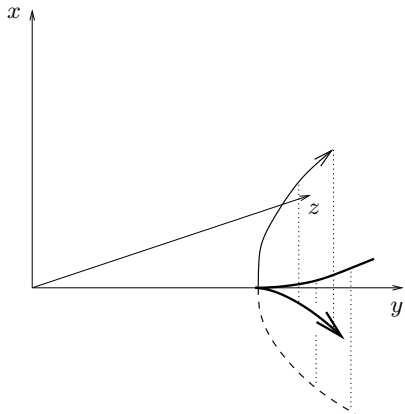}
\extralabel <-0.3cm,1.4cm> {$y$}
\extralabel <-5.4cm,5.3cm> {$x$}
\extralabel <-1.4cm,2.7cm> {$z$}
\endrelabelbox}
\caption{Contribution of a cusp to $\rot (\gamma )$.}
\label{figure:leftnegcusp}
\end{figure}

At a left cusp oriented downwards, the tangent vector to $\gamma$, expressed in
terms of $e_1,e_2$, changes from having a negative component in the
$e_2$-direction to a positive one, i.e.\ such a cusp yields a
positive contribution to $\rot (\gamma )$. Analogously, one sees that
a right cusp oriented upwards gives a negative contribution to
the rotation number. This proves the formula
$\rot (\gamma )=\lambda_- - \rho_+$.
The second expression for the rotation number is obtained by
counting crossings through $-e_1$ instead; the third expression
is found by averaging the first two.
\end{proof}
\section{Proof of the Whitney--Graustein theorem}
First we give a classification of regular closed Legendrian curves up to
Legendrian regular homotopy.

\begin{prop}
Legendrian regular homotopy classes of regular closed Legendrian curves
$\gamma\co S^1\rightarrow (\R^3,\xi )$
are in one-to-one correspondence with the integers, the
correspondence being given by $[\gamma ]\mapsto\rot (\gamma )$.
\end{prop}

\begin{proof}
With the help of either of the two foregoing propositions one
can construct a regular closed Legendrian curve $\gamma$
with $\rot (\gamma )$ equal to any prescribed integer. Thus,
we need only show that two regular closed Legendrian curves
$S^1\rightarrow (\R^3,\xi )$
with the same rotation number are Legendrian regularly homotopic.

In the front projection of the Legendrian
immersion~$\gamma$, left and right cusps alternate. We label the up-cusps
with $+$ and the down-cusps with~$-$. Up to Legendrian regular homotopy,
$\gamma$ is completely determined by the sequence of these labels,
starting at a right-cusp, say, and going once around~$S^1$. 
This can be seen by homotoping $\gamma_{\mathrm F}$ so that
all left cusps come to lie on the line
$\{ y=0\}$ and all right cusps on the line $\{ y=1\}$, say. The cusps on
either line can be shuffled by further homotopies; in particular,
they may be arranged in the same order along these lines in
which they are traversed along the closed Legendrian
curve. Figure~\ref{figure:model}
shows a standard model for a front $\gamma_{\mathrm F}$ containing
cusps of one sign only.

\begin{figure}[h]
\centerline{\relabelbox\small
\epsfxsize 5cm \epsfbox{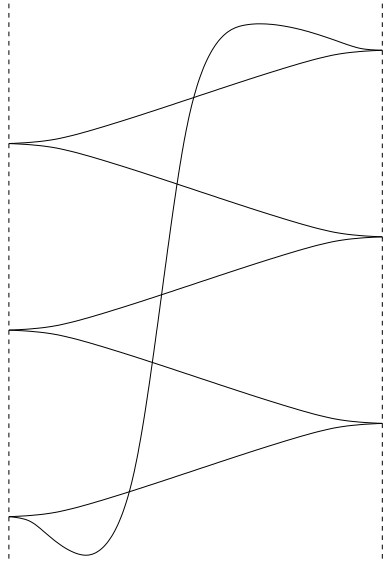}
\endrelabelbox}
\caption{A front with cusps of one sign only.}
\label{figure:model}
\end{figure}

Moreover, a pair $+-$ or $-+$ can be cancelled from this sequence by a
Legendrian regular homotopy; locally this is in fact achieved by
a Legendrian isotopy, i.e.\ a regular homotopy not
creating self-intersections: the so-called
first Legendrian Reidemeister move
(see Figure~\ref{figure:reidemeister1}; there is an analogous move with
the picture rotated by $180^{\circ}$).

\begin{figure}[h]
\centerline{\relabelbox\small
\epsfxsize 6cm \epsfbox{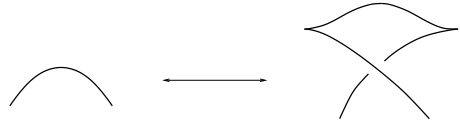}
\endrelabelbox}
\caption{The first Legendrian Reidemeister move.}
\label{figure:reidemeister1}
\end{figure}

Therefore, this sequence of labels can be reduced to a sequence containing
only plus or only minus signs, or to one of the sequences $(+,-)$,
$(-,+)$; see Figure~\ref{figure:reghomotopy3} for
an example. The formula $\rot (\gamma )=(c_- - c_+)/2$
shows that there are the following
possibilities: if $\rot (\gamma)$ is positive (resp.\ negative),
we must have a sequence of $2\rot (\gamma)$ minus (resp.\ plus) signs;
if $\rot (\gamma )=0$, we must have the sequence $(+,-)$ or $(-,+)$.
The proof is completed by observing that these last two sequences
correspond to Legendrian isotopic knots: use a first
Reidemeister move as in Figure~\ref{figure:reidemeister1}, followed
by the inverse of the rotated move.
\end{proof}

\begin{figure}[h]
\centerline{\relabelbox\small
\epsfxsize 8cm \epsfbox{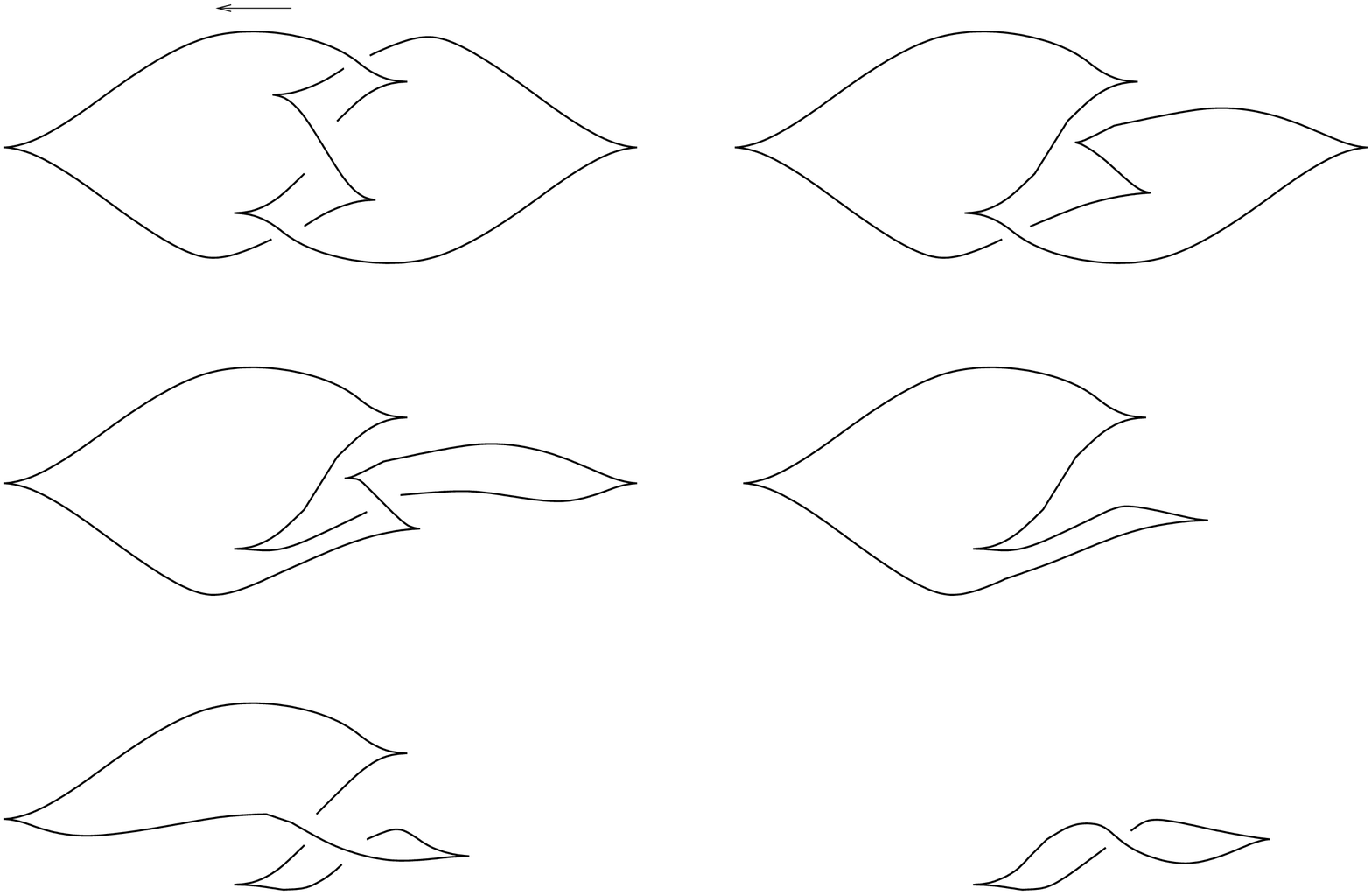}
\extralabel <-8.0cm,5.1cm> {(i)}
\extralabel <-8.0cm,3.2cm> {(iii)}
\extralabel <-8.0cm,1.1cm> {(v)}
\extralabel <-3.4cm,5.1cm> {(ii)}
\extralabel <-3.4cm,3.2cm> {(iv)}
\extralabel <-3.4cm,1.1cm> {(vi)}
\endrelabelbox}
\caption{An example of a Legendrian regular homotopy.}
\label{figure:reghomotopy3}
\end{figure}

\begin{rem}
Self-tangencies in the front projection $\gamma_{\mathrm F}$ correspond to
self-intersections of the Legendrian curve~$\gamma$, since the
negative slope of $\gamma_{\mathrm F}$ gives the $x$-component of~$\gamma$.
Therefore, as we pass such a self-tangency in the moves of
Figure~\ref{figure:reghomotopy3}, we effect a crossing change.
With the orientation indicated in the figure, this
example has $\rot (\gamma ) = -1$.
\end{rem}

\begin{proof}[Proof of Theorem \ref{thm:WG}]
Again we only have to show that
two regular closed curves $\ogamma_0,\ogamma_1\co S^1\rightarrow\R^2$ 
(where we think of $\R^2$ as the $(x,y)$-plane) with
$\rot (\ogamma_0)=\rot (\ogamma_1 )$ are regularly homotopic.

After a regular homotopy we may assume that the $\ogamma_i$ satisfy
the area condition $\oint_{\ogamma_i}x\, dy =0$ and thus lift to
regular {\em closed} Legendrian curves $\gamma_i\co S^1\rightarrow
(\R^3,\xi )$ with --- by Proposition~\ref{prop:rot-rot} ---
$\rot (\gamma_i)=\rot (\ogamma_i)$. By the preceding proposition,
$\gamma_0$ and $\gamma_1$ are Legendrian regularly homotopic. The
Lagrangian projection of this homotopy gives a regular homotopy
between the curves $\ogamma_0$ and~$\ogamma_1$, since --- as
pointed out in Section~\ref{section:Legendrian} --- the Lagrangian
projection of a regular Legendrian curve is regular.
\end{proof}

\end{document}